\documentclass[12pt,a4paper]{amsart}
\usepackage{amsfonts}
\usepackage{amsthm}
\usepackage{amsmath}
\usepackage{amscd}
\usepackage[latin2]{inputenc}
\usepackage{t1enc}
\usepackage[mathscr]{eucal}
\usepackage{indentfirst}
\usepackage{graphicx}
\usepackage{graphics}
\usepackage{pict2e}
\usepackage{epic}
\usepackage{color}
\numberwithin{equation}{section}
\usepackage[margin=2.9cm]{geometry}

\theoremstyle{plain}
\newtheorem{Th}{Theorem}[section]
\newtheorem{Lemma}[Th]{Lemma}

\newtheorem{Prop}[Th]{Proposition}

 \theoremstyle{definition}
\newtheorem{Def}[Th]{Definition}

\newtheorem{Rem}[Th]{Remark}
\newtheorem{?}[Th]{Problem}
\newtheorem{Cons}[Th]{Construction}

\newcommand{\ovl}{\overline}
\newcommand{\Pm}{\operatorname{pm}}
\newcommand{\C}{\mathbb{C}}
\begin{document}

\title{Matchings in Benjamini--Schramm convergent graph sequences}

\author[M. Ab\'ert]{Mikl\'os Ab\'ert}

\address{Alfr\'ed
R\'enyi Institute of Mathematics \\ H-1053 Budapest \\ Re\'altanoda
u. 13-15. \\ Hungary}

\email{abert@renyi.hu}

\author[P. Csikv\'ari]{P\'{e}ter Csikv\'{a}ri}

\address{Massachusets Institute of Technology \\ Department of Mathematics \\
Cambridge MA 02139 \&  E\"{o}tv\"{o}s Lor\'{a}nd University \\ Department of Computer
Science \\ H-1117 Budapest
\\ P\'{a}zm\'{a}ny P\'{e}ter s\'{e}t\'{a}ny 1/C \\ Hungary}

\email{peter.csikvari@gmail.com}

\author[P. E. Frenkel]{P\'eter E. Frenkel}

\address{E\"{o}tv\"{o}s Lor\'{a}nd University \\ Department of Algebra and
Number Theory \\ H-1117 Budapest
\\ P\'{a}zm\'{a}ny P\'{e}ter s\'{e}t\'{a}ny 1/C \\ Hungary}

\email{frenkelp@cs.elte.hu}

\author[G. Kun]{G\'abor Kun}

\address{E\"{o}tv\"{o}s Lor\'{a}nd University \\ Department of Computer
  Science \\ H-1117 Budapest
\\ P\'{a}zm\'{a}ny P\'{e}ter s\'{e}t\'{a}ny 1/C \\ Hungary}

\email{kungabor@cs.elte.hu}

\thanks{The first and third authors were partially supported by ERC Consolidator Grant 648017. The first three authors were partially supported by the Hungarian National Foundation for Scientific Research (OTKA), grant no.\ K109684. The second author was partially supported by the National
Science Foundation under grant no.\ DMS-1500219 and the Hungarian National Foundation for
Scientific Research (OTKA), grant no.\ K81310.  The fourth author  is partially supported by the
Hungarian National Foundation for Scientific Research (OTKA) Grant
No.\ PD 109731, by the J\'anos Bolyai Scholarship of the Hungarian
Academy of Sciences, by the Marie Curie IIF Fellowship Grant No. 627476
and by ERC Advanced Research Grant No.\ 227701.
All authors are partially supported by MTA R\'enyi
"Lend\"ulet" Groups and Graphs Research Group.}



\begin{abstract} We introduce the {\emph{matching measure}} of a finite graph as the uniform distribution on the roots of the matching polynomial of the graph. We analyze the asymptotic behavior of the matching measure for graph sequences with bounded degree.

A graph parameter is said to be estimable if it converges along every Benjamini--Schramm convergent sparse graph sequence.
We prove that the normalized logarithm of the number of matchings is estimable. We also show that the analogous statement for \emph{perfect} matchings already fails for $d$--regular bipartite graphs for any fixed $d\ge 3$.  The latter result relies on analyzing the probability that a randomly chosen perfect matching contains a particular edge.

However, for any sequence of $d$--regular bipartite graphs converging to the $d$--regular tree, we prove that the normalized logarithm of the number of perfect matchings converges. This applies to random $d$--regular bipartite graphs. We show that the limit equals to the exponent in Schrijver's lower bound on the number of perfect matchings.

Our analytic approach also yields a short proof for the Nguyen--Onak (also Elek--Lippner) theorem saying that the matching ratio is estimable. In fact, we prove the slightly stronger result that the independence ratio is estimable for claw-free graphs.

\end{abstract}

\maketitle

\section{Introduction}

In this paper, we study the asymptotic behavior of the number of matchings
and perfect matchings for  Benjamini--Schramm convergent sequences of finite
graphs. Benjamini-Schramm convergence was introduced in \cite{bensch} and has
been under intense investigation since then.

For a finite graph $G$, a finite rooted graph $\alpha$ and a positive integer
$r$, let $\mathbb{P}(G,\alpha,r)$ be the probability that the $r$-ball
centered at a uniform random vertex of $G$ is isomorphic to $\alpha$ (as a
rooted graph). A sequence of finite graphs $(G_{n})$ is \emph{sparse} if the set of
degrees of vertices in $G_{n}$ ($n\geq1$) is bounded. A sparse graph sequence
$(G_{n})$ is \emph{Benjamini--Schramm convergent} if for all
finite rooted graphs $\alpha$ and $r>0$, the probabilities $\mathbb{P}%
(G_{n},\alpha,r)$ converge. This means that one cannot distinguish $G_{n}$
and $G_{n^{\prime}}$ for large $n$ and $n^{\prime}$ by sampling them at a
random vertex with a fixed radius of sight.

The graph parameter $p(G)$ is \emph{estimable} in a class $\mathcal{C}$ of
finite graphs if the sequence $p(G_{n})$ converges for all Benjamini--Schramm
convergent sparse graph sequences $(G_{n})$ in $\mathcal{C}$. When
$\mathcal{C}$ is the class of all finite graphs, we simply say that $p(G)$ is
estimable.

Let $v(G)$, $e(G)$, $\mathbb{M}(G)$, and $\operatorname{pm}(G)$ stand for the
number of vertices, edges, matchings, and perfect matchings in the graph $G$,
respectively. We write $\nu(G)$ and $\alpha(G)$ for the maximal size of a
matching, respectively an independent vertex set in $G$.

\subsection{Estimable matching parameters}

There are several examples of seemingly `global' graph parameters that turn
out to be estimable. A striking example is the following theorem of R. Lyons \cite{lyons}.

\begin{Th}
[R.\ Lyons]Let $\tau(G)$ denote the number of spanning trees in
the graph $G$. Then the \emph{tree entropy per site}
\[
\frac{\ln\tau(G)}{v(G)}%
\]
is estimable in the class of connected graphs.
\end{Th}

Our first result shows that a similar statement is true for the number of
matchings.

\begin{Th}
\label{entropy-intro} The \emph{matching entropy per site}
\[
\frac{\ln\mathbb{M}(G)}{v(G)}%
\]
is estimable.
\end{Th}

This will be proved as a part of Theorem~\ref{wc}.
 For the proof, we apply the machinery developed by the first three authors and T. Hubai in the papers \cite{abert_hubai,csi}. In particular, results in \cite{csi} show that if
$f(G,x)$ is a graph polynomial satisfying certain conditions and $\rho_{G}$ is
the uniform distribution on the roots of $f(G,x)$, then for every fixed $k$,
the graph parameter
\[
\int z^{k}\mathrm{d}\rho_{G}%
\]
is estimable.

When considering the matching polynomial as $f(G,x)$, we get the definition of the {\it matching measure} and that the matching measure weakly converges for Benjamini--Schramm convergent sequences of graphs.
This leads to Theorem \ref{entropy-intro}.
Note that a modification of  the algorithm `CountMATCHINGS' in \cite{BGK+} yields an alternative proof of Theorem~\ref{entropy-intro}.

Considering the \emph{independence
polynomial} as $f(G,x)$, however, also yields an extension of the following theorem of
H.\ Nguyen and K.\ Onak \cite{ngu} (independently proved by G.\ Elek and
G.\ Lippner \cite{ele}).

\begin{Th}
[Nguyen--Onak and Elek--Lippner]\label{NOEL} The
\emph{matching ratio} ${\nu(G)}/{v(G)}$ is estimable.
\end{Th}

A graph is \emph{claw-free} if it does not contain the complete bipartite
graph $K_{1,3}$ as an induced subgraph. Our extension is the following.

\begin{Th}
\label{claw--free} The \emph{independence ratio} $\alpha(G)/v(G)$ is estimable
in the class of claw-free graphs.
\end{Th}

This will be proved as a part of Theorem~\ref{wcind}. By the following
argument, Theorem~\ref{claw--free} indeed extends Theorem~\ref{NOEL}. The
\emph{line graph} $L(G)$ of the graph $G$ has vertex set $V(L(G))=E(G)$, and
$e,f\in E(G)$ are adjacent in $L(G)$ if they share an endpoint in $G$.
Trivially $\nu(G)=\alpha(L(G))$, so we have
\[
\frac{\nu(G)}{v(G)}=\frac{\nu(G)}{e(G)}\cdot\frac{e(G)}{v(G)}=\frac
{\alpha(L(G))}{v(L(G))}\cdot\frac{e(G)}{v(G)}%
\]
that is, the matching ratio of $G$ equals the independence ratio of $L(G)$
times the edge density of $G$. The edge density is clearly estimable. Using
that line graphs are claw-free, and that $(L(G_{n}))$ is Benjamini--Schramm
convergent if $(G_{n})$ is, Theorem \ref{claw--free} implies Theorem
\ref{NOEL}.

Note that the independence ratio is not estimable in general. Indeed, random
$d$-regular graphs and random $d$-regular \emph{bipartite} graphs converge to
the same object, the $d$-regular tree, but by a result of B. Bollob\'as \cite{bollobas},
the independence ratio of a sequence of random $d$-regular graphs is bounded away
from $1/2$ a.s.

\subsection{Matchings and perfect matchings in graphs with essentially large
girth}

The girth $g(G)$ of the graph $G$ is defined to be the length of the shortest
cycle in $G$. If $(G_{n})$ is a sequence of  $d$--regular graphs with
$g(G_{n})\rightarrow\infty$, then $(G_{n})$ Benjamini--Schramm converges,
since every $r$-ball of $G_{n}$ will be isomorphic to the $r$-ball of the
$d$-regular tree for large enough $n$. More generally, we say that $(G_{n})$
is of \emph{essentially large girth }(or converges to the $d$-regular
tree), if for any fixed $k$, the number of $k$--cycles in $G_{n}$ is
$o(v(G_{n}))$ as $n\rightarrow\infty$. Important examples are sequences of
random $d$--regular graphs and bipartite graphs.

The following theorem is the main result of this paper.

\begin{Th}
\label{entropy-ggirth} Let $(G_{n})$ be a sequence of $d$--regular graphs with
essentially large girth. Then the following hold.

\begin{itemize}
\item[(a)] We have
\[
\lim_{n\rightarrow\infty}\frac{\ln\mathbb{M}(G_{n})}{v(G_{n})}=\frac{1}{2}\ln
S_{d},
\]
where
\[
S_{d}=\frac{1}{\xi^{2}}\left(  \frac{d-1}{d-\xi}\right)  ^{d-2},\qquad
\xi=\frac{2}{1+\sqrt{4d-3}}.
\]
In particular, $S_{3}=16/5$.

\item[(b)] For the number of perfect matchings $\operatorname{pm}(G_{n})$, we
have
\[
\limsup_{n\to\infty}\frac{\ln\operatorname{pm}(G_{n})}{v(G_{n})}\leq\frac
{1}{2}\ln\left(  \frac{(d-1)^{d-1}}{d^{d-2}}\right) .
\]

\item[(c)] If, in addition, the graphs $(G_{n})$ are bipartite, then
\[
\lim_{n\to\infty}\frac{\ln\operatorname{pm}(G_{n})}{v(G_{n})}=\frac{1}{2}%
\ln\left(  \frac{(d-1)^{d-1}}{d^{d-2}}\right) .
\]

\end{itemize}
\end{Th}

Theorem \ref{entropy-ggirth} is related to the following famous result of A. Schrijver \cite{sch1}.

\begin{Th}
[Schrijver]\label{Schrijver} For any $d$--regular bipartite graph
$G$ on  $v(G)=2\cdot n$ vertices, we have
\[
\operatorname{pm}(G)\geq\left( \frac{(d-1)^{d-1}}{d^{d-2}}\right) ^{n}.
\]

\end{Th}

In other words, for a $d$--regular bipartite graph $G$ we have
\[
\frac{\ln\operatorname{pm}(G)}{v(G)}\geq\frac{1}{2}\ln\left( \frac
{(d-1)^{d-1}}{d^{d-2}}\right) .
\]
For $d=3$, this theorem was proved by M. Voorhoeve \cite{vor}. Then A. Schrijver
\cite{sch1} proved it for every $d$. A very elegant new proof was given by
L. Gurvits \cite{gur}. For a simplified version of Gurvits's proof see \cite{L-S}.

A. Schrijver and W. G. Valiant proved in \cite{sch2} that the exponent
\[
\frac{(d-1)^{d-1}}{d^{d-2}}%
\]
cannot be improved by showing that for a random $d$--regular bipartite
multigraph the statement is asymptotically sharp. I.\ Wanless noticed in
\cite{wan} that the same holds if we do not allow multiple edges.  More refined results are proved by B. Bollob\'as and B. D. McKay \cite{bolmc}. Theorem
\ref{entropy-ggirth} shows that the only thing that is relevant here about random $d$--regular
bipartite graphs is that they are of essentially large girth. In particular,
this makes the sharpness statement of A. Schrijver and W. G. Valiant constructive, as
Part (c) of Theorem~\ref{entropy-ggirth} now allows us to \emph{construct}
bipartite graphs with an asymptotically minimal number of perfect matchings.
Indeed, we simply have to construct $d$--regular bipartite graphs with large
girth, which is known to be possible in various ways. See for instance
\cite{erd} for constructing $d$--regular graphs with large girth and note that
if $G$ is $d$--regular, then the (weak) direct product $G\times K_{2}$ is
$d$--regular bipartite and satisfies $g(G\times K_{2})\geq g(G)$.

\subsection{Perfect matchings with no assumption on girth}

It is natural to ask whether Theorem \ref{entropy-intro} holds for the number
$\operatorname{pm}(G)$ of perfect matchings rather than the number
$\mathbb{M}(G)$ of all matchings. It is easy to see that in the class of all
graphs, the perfect matching entropy per site
\[
\frac{\ln\operatorname{pm}(G)}{v(G)}%
\]
is not estimable. Indeed, one can consider a large graph with many perfect
matchings and then add an isolated vertex to it. Then the two graphs are very
close in local statistics, but the latter graph has no perfect matching. This
is of course a quite cheap example. On the other hand, it turns out that the
situation does not get much better even for the class of $d$--regular bipartite graphs.

\medskip\noindent\textbf{Notation.} Given a finite graph admitting at least
one perfect matching, and an edge $e$, let $p(e)$ denote the probability that
$e$ is contained in a uniformly chosen random perfect matching of the graph.
\medskip

We shall prove that perfect matchings can get quite unevenly distributed.

\begin{Th}
\label{construction} For any integer $d\ge3$, there exists a constant $0<c<1$
such that  for any positive integer $n\ge d$ there exists a $d$--regular
bipartite  simple graph on $2\cdot n$ points with an edge $e$ such that
\[
p(e)>1-c^{n}.
\]

\end{Th}

This leads to the following.

\begin{Th}
\label{perfect_matching} Fix $d\ge3$. The perfect matching entropy per site
\[
\frac{\ln\operatorname{pm} (G)}{v(G)}%
\]
is not estimable in the class of $d$--regular bipartite simple graphs.
\end{Th}

The phenomenon in Theorem~\ref{perfect_matching} does not occur for
\emph{expander} bipartite graphs. Indeed, it can be deduced from Corollary 1 of D.\ Gamarnik and D.\ Katz \cite{GK} that
for any $\delta>0$,  the perfect matching entropy per site
is estimable for $d$--regular bipartite $\delta$-expander graphs. We thank D. Gamarnik for pointing this out for us.
The phenomenon in Theorem~\ref{construction} cannot occur either for
expander bipartite graphs: an edge probability cannot be exponentially close to 1. In fact, we shall  prove the following stronger statement about edge probabilities.

\begin{Th}
\label{expander}Let $n\geq2$, $\delta>0$, let $G$ be a $\delta$-expander
bipartite graph of maximum degree $d$ on $2\cdot n$ vertices, and $e$ an edge
of $G$. Assume that $G$ admits a perfect matching.
Then
\[
p(e)\geq\frac{1}{d}n^{-2\cdot{\ln(d-1)}/{\ln(1+\delta)}}.
\]

\end{Th}

\subsection{Structure of the paper}

The paper is organized as follows. In Section~\ref{claw-free}, we gather a few
known results about the independence polynomial and prove Theorem~\ref{wcind}
about independent vertex sets in claw-free graphs. The reader only interested
in matchings can skip this section without harm.

In Section~\ref{matching_measure}, we gather some known results about the
matching polynomial and prove Theorem~\ref{wc} about matchings.
In Section~\ref{GIRTH}, we prove Theorem~\ref{entropy-ggirth} about matchings
in essentially large girth graphs. In Section~\ref{perfect_entropy}, we prove the negative
results: Theorem~\ref{construction} and Theorem~\ref{perfect_matching}. In
Section~\ref{expander_helps} we give the proof of Theorem~\ref{expander} about
expanders. Finally, in Section~\ref{we don't know} we pose some open problems.

\section{Independent sets in claw-free graphs} \label{claw-free}
\subsection{The independence polynomial}
\begin{Def} Let $G$ be a graph on $v(G)$ vertices. Let ${\alpha(G)}$ be the maximal size of an independent vertex set,  and let $i_k(G)$ denote the number of independent sets
  of size $k$. Then the \it independence  polynomial \rm  $I(G,x)$ is defined as follows:
$$I(G,x)=\sum_{k=0}^{\alpha(G)}i_k(G)x^{k}.$$
Note that $i_0(G)=1$.
The \it independence measure \rm $\sigma_G$  is defined as

\begin{equation*}
\sigma _{G}=\frac{1}{v(G)}\sum\limits_{\lambda }r(G, \lambda )\delta _{\lambda }
\end{equation*}%
where $\lambda $ runs through the roots of $I(G,x)$, $r(G, \lambda )$ is the multiplicity of $\lambda $ as a root of $I(G,x)$ and $\delta _{\lambda }$ denotes the Dirac measure at $\lambda $.
\end{Def}
Note that unless $G$ is the empty graph, the independence measure is not a probability measure.

Many graph parameters related to independent sets can be read off from 
 the independence measure.
\begin{Def}
For a finite graph $G$ let $\kappa_G$ denote the size of a uniform random independent subset of $G$.
\end{Def}
So $\kappa_G$ is a random variable depending on $G$.

Besides the number of all independent sets $$\mathbb I(G)=I(G,1)=\sum_{k=0}^{\alpha(G)}i_k(G)$$
we shall also be interested in the expected size
\[\mathbb E\kappa_G=\frac{\sum ki_k(G)}{\sum i_k(G)}\]
and the variance
\[{\mathbb D^2\kappa_G} =\frac{\sum (k-\mathbb E\kappa_G)^2i_k(G)}{\sum i_k(G)}.\]

\begin{Prop}\label{indpar} \begin{enumerate}
\item[(a)]
For the \emph{independent set entropy per vertex}, we have  \[\frac{\ln \mathbb I(G)}{v(G)}=\int\ln|1-x|\mathrm d\sigma_G(x).\]

\item[(b)] The normalized expected value \begin{align*}
\frac{\mathbb E\kappa_G}{v(G)} =
\int\frac{\mathrm 1}{1-x} d\sigma_G(x)
.
\end{align*}

\item[(c)]
The normalized variance
\[\frac{\mathbb D^2\kappa_G}{v(G)} =\int \frac{-x}{(1-x)^2}\mathrm d\sigma_G(x).\]

\item[(d)]
The  independence ratio of $G$ equals \[\alpha (G)/v(G)=\sigma_G(\C).\]
\end{enumerate}
\end{Prop}

\begin{proof} \begin{enumerate}
\item[(a)]
We have \[\mathbb I(G)=\sum_{k=0}^{\alpha(G)}i_k(G)=I(G,1)
.\] Thus,  \[\frac{\ln \mathbb I(G)}{v(G)}=\frac{\ln |I(G,1)|}{v(G)}=\int\ln|1-x|\mathrm d\sigma_G(x).\]

\item[(b)] We have \begin{align*}
\frac{\mathbb E\kappa_G}{v(G)} =\frac1{v(G)}
\frac{\sum ki_k(G)}{\sum i_k(G)}=
\frac1{v(G)}\frac{I'(G,1)}{I(G,1)}=
\int\frac{\mathrm d\sigma_G(x)}{1-x}
.
\end{align*}

\item[(c)]
Let $\lambda_1$, \dots, $\lambda_{\alpha(G)}$
be the roots of the polynomial $I(G,x)$.
We have \[\left(\sum\frac1{1-\lambda_i}\right)^2=\sum\frac1{(1-\lambda_i)^2}+\sum_{i\ne j}\frac1{(1-\lambda_i)(1-\lambda_j)},\]
i.e.,
\[
\left(v(G)\int\frac{\mathrm d\sigma_G(x)}{1-x}\right)^2=v(G)\int\frac{\mathrm d\sigma_G(x)}{(1-x)^2}+\frac{I''(G,1)}{I(G,1)},\] in other words,
\[(\mathbb E\kappa_G)^2=v(G)\int\frac{\mathrm d\sigma_G(x)}{(1-x)^2}+\mathbb E(\kappa_G(\kappa_G-1)),\]
so
\[\mathbb D^2\kappa_G=\mathbb E\kappa_G^2-(\mathbb E\kappa_G)^2=\mathbb E\kappa_G-v(G)\int\frac{\mathrm d\sigma_G(x)}{(1-x)^2},\] and the claim follows using statement (b).

\item[(d)] Obvious from the definition.
\end{enumerate}
\end{proof}

To study the behaviour of the  independence measure in a convergent graph sequence, we need to have some control on the location of the roots in terms of the greatest degree in a graph.
It follows from Dobrushin's lemma
that all roots  of $I(G,x)$  have  absolute value greater than \[\beta := \frac{\exp(-1)}{d+1},\] where $d$ is the
greatest degree in $G$, cf.\ \cite[Corollary 5.10]{csi}.

The following lemma has its roots in \cite{luck}, see also \cite{abthvi}.

\begin{Lemma}\label{ize} For all $R>1$ we have
\begin{equation}\label{delta}{\sigma_G(|x|\ge R)}\le\frac{\ln(1/\beta)}{\ln R}.
\end{equation}
\end{Lemma}

\begin{proof}
The product of the roots of $I(G,x)$ is, in absolute value, equals $1/i_{\alpha(G)}(G)\le 1$.
Thus, for any $R>1$ we have
\[R^{\sigma_G(|x|\ge R)}\beta\le 1,\]
which proves the lemma. \end{proof}

\subsection{Claw-free graphs}
When $G$ is claw-free, all roots of $I(G,x)$ are real by \cite{chu}.

The following theorem deals with the behaviour of the independence  polynomial in
Benjamini--Schramm convergent  sequences of claw-free graphs.

\begin{Th} \label{wcind}
Let $(G_n)$ be  a Benjamini--Schramm convergent
 claw-free graph sequence with absolute degree bound  $d$.  Set   $H=(-\infty,-\beta]$.
Then the sequence of independence  measures  $\sigma_n=\sigma_{G_n}$
converges weakly to a  measure $\sigma$ on $H$. As $n\to\infty$, we have
\[\frac{\ln \mathbb I(G_n)}{v(G_n)}\to \int_H \ln (1-x)\mathrm d\sigma(x),\]
\[ \frac{\mathbb E\kappa_{G_n}}{v(G_n)}
\to\int_H\frac{\mathrm d\sigma(x)}{1-x},\]

 \[\frac{\mathbb D^2\kappa_{G_n}}{v(G_n)}
\to\int_H\frac{-x}{(1-x)^2}\mathrm d\sigma(x),\]
 and
\[\alpha(G_n)/v(G_n)\to\sigma(H).\]
In particular,   $(\ln\mathbb I)/v$, $\mathbb E\kappa/v$, $\mathbb D^2\kappa/v$  and $\alpha/v$ are estimable graph parameters for claw-free graphs.
\end{Th}

Note that this recovers Theorem~\ref{claw--free}.

\begin{proof}
We consider the graph polynomial
\[f(G,x)=
x^{v(G)}I(G, 1/x)=\sum_{k=0}^{\alpha(G)} i_k(G)x^{v(G)-k}.\]
Let $\tau=\tau_G$ be the probability measure of uniform distribution on the roots of $f(G,x)$. For $G$
claw-free with greatest degree $d$, this measure is supported on
$K=[-1/\beta,0]$. The graph polynomial $f(G,x)$ is monic of degree $v(G)$, and it is multiplicative with respect to disjoint union of graphs because \[i_k(G)=\sum_{k_1+k_2=k}i_{k_1}(G_1)i_{k_2}(G_2)\] whenever $G=G_1\cup G_2$ is a  disjoint union. The coefficient $i_k(G)$ of $x^{v(G)-k}$ is the number of induced
subgraphs of $G$ that are isomorphic to the empty graph on $k$ points. By
the well-known and easy \cite[Fact 3.2]{csi}, this can be expressed as a
finite linear combination
\[i_k(G)=\sum_H c_{H,k}H(G),\]
where $H(G)$ is the number of (not necessarily induced) subgraphs of $G$
that are isomorphic to $H$.


 Note that  $\mathbb C\setminus K$ is connected and $K$ has empty  interior (as a  subset of $\mathbb C$).  By  \cite[Theorem 4.6(a)]{csi}, it follows that the sequence $\int_K g\mathrm d\tau_{G_n}$  converges for all continuous $g:K\to \mathbb R$.
Thus, the sequence $\int_H g\mathrm d\sigma_n$  converges for any continuous $g:H\to \mathbb R$ that tends to zero at $-\infty$.  Using \eqref{delta}, we see that this last decay assumption may be dropped, so $\sigma_n$ converges weakly.

Since $1\not\in H$, we have
\[\frac{\ln \mathbb I(G_n)}{v(G_n)}=\int_H \ln (1-x)\mathrm d\sigma_{n}\to\int_H  \ln (1-x) \mathrm d\sigma .\]

The other statements  follow from Proposition~\ref{indpar} the same way.
\end{proof}

\section{Matching polynomial and Benjamini--Schramm convergence} \label{matching_measure}

\begin{Def} Let $G$ be a graph on $v(G)=v$ vertices and let $m_k(G)$ denote the number of matchings
  of size $k$. Then the \it matching polynomial \rm  $\mu(G,x)$ is defined as follows:
$$\mu(G,x)=\sum_{k=0}^{\lfloor v/2\rfloor} (-1)^km_k(G)x^{v-2k}.$$
Note that $m_0(G)=1$.
The \it matching measure \rm $\rho_G$  is defined as

\begin{equation*}
\rho_{G}=\frac{1}{v(G)}\sum\limits_{\lambda }r(G, \lambda )\delta _{\lambda }
\end{equation*}%
where $\lambda $ runs through the roots of $\mu(G,x)$, $r(G, \lambda )$ is the multiplicity of $\lambda $ as a root of $\mu(G,x)$ and $\delta _{\lambda }$ denotes the Dirac measure at $\lambda $.

Let $\gamma_G$ denote the number of edges in a uniform random matching of $G$.
\end{Def}

\begin{Rem}\label{trans} Let $L(G)$ be the line graph of $G$.  Then $m_k(G)=i_k(L(G))$, so  \[\mu(G,x)=x^{v(G)}I(L(G), -1/x^2).\] Therefore, if $x$ runs over the nonzero roots of $\mu(G,x)$, then $-1/x^2$ runs over the roots of  $I(L(G),x)$  twice.  Remember that $\sigma_{L(G)}$ assigns weight $1/v(L(G))=1/e(G)$ times the multiplicity to each root of   $I(L(G),x)$, while $\rho_G$ assigns weight $1/v(G)$ times the multiplicity to each root of   $\mu(G,x)$. Thus, for any function $g$ defined on the roots of $I(L(G),x)$, we have
\[2e(G)\int g(x)\mathrm d\sigma_{L(G)}(x)=v(G)\int _{x\ne 0}g(-1/x^2)\mathrm d\rho_G(x).\]
\end{Rem}

Using Remark~\ref{trans}, almost all results in this section follow from their counterparts in Section~\ref{claw-free}.
Converting the results requires about the same amount of work as redoing the proofs. In some cases we will do the latter for the convenience of the reader who is only interested in matchings and therefore skipped Section~\ref{claw-free}.

The fundamental theorem for the matching polynomial is the following.

\begin{Th}[Heilmann and Lieb \cite{hei}]

\begin{itemize}

\item[(a)]
The roots of the matching polynomial
$\mu(G,x)$ are real.

\item[(b)] If  $d\ge 2 $ is an upper bound for all degrees in $G$,
then  all roots of $\mu(G,x)$ have absolute value $\le 2\sqrt{d-1}$.
\end{itemize}
\end{Th}

Many graph parameters related to matchings can be read off from 
 the matching measure.
 Besides the number $$\mathbb M(G)=\sum_{k=0}^{\alpha(G)}m_k(G)$$ of all matchings, we shall also be interested in the
expectation \[\mathbb E\gamma_G=\frac{\sum km_k(G)}{\sum m_k(G)}\]
and also in the variance
\[{\mathbb D^2\gamma_G} =\frac{\sum (k-\mathbb E_G\gamma_G)^2m_k(G)}{\sum m_k(G)}.\]

\begin{Prop}\label{matchpar}\begin{itemize}\item[(a)] For the \emph{matching entropy per vertex}, we have \[\frac{\ln \mathbb M(G)}{v(G)}
=\frac12\int\ln (1+x^2)\mathrm d\rho_G(x).\]
The normalized expected value of $\gamma_G$ equals \begin{align*}
\frac{\mathbb E_G\gamma_G}{v(G)} =
\frac12\int\frac{x^2}{1+x^2}\mathrm d\rho_G(x).
\end{align*}

The normalized variance of $\gamma_G$ equals \[\frac{\mathbb D^2\gamma_G}{v(G)} 
=\frac12\int \frac{x^2}{(1+x^2)^2}\mathrm d\rho_G(x).\]

The matching ratio equals \[\frac{\nu(G)}{v(G)}=\frac{1-\rho_G(\{0\})}2.\]

\item[(b)]
For the \emph{perfect matching entropy per vertex}, we have \[\frac{\ln \Pm (G)}{v(G)}=
\int\ln|x|\mathrm d\rho_G(x)
.\]

\end{itemize}
\end{Prop}

\begin{proof}

\begin{itemize}\item[(a)]  All statements follow from Proposition~\ref{indpar} and Remark~\ref{trans}. However, we give direct proofs for the first two statements.

The number of matchings in $G$ is \[\mathbb M(G)=\sum_{k=0}^{\lfloor v/2\rfloor}m_k(G)=|\mu(G,\sqrt{-1})|
.\] Thus,  \[\frac{\ln \mathbb M(G)}{v(G)}=\frac{\ln |\mu(G,\sqrt{-1})|}{v(G)}=\int\ln|\sqrt{-1}-x|\mathrm d\rho_G(x)
=\frac12\int\ln (1+x^2)\mathrm d\rho_G(x)
.\]

We have \begin{align*}
\frac{\mathbb E\gamma_G}{v(G)} =
\frac1{v(G)}\frac{\sum km_k(G)}{\sum m_k(G)}=
\frac12\left(1-\frac1{v(G)}\frac{|\mu'(G,\sqrt{-1})|}{|\mu(G,\sqrt{-1})|}\right)=\\=
\frac12\left(1-\sqrt{-1}\int\frac{\mathrm d\rho_G(x)}{\sqrt{-1}-x}\right)=
\frac12\int\frac{x^2}{1+x^2}\mathrm d\rho_G(x)
.
\end{align*}

The statement for the  normalized variance is straightforward from Proposition~\ref{indpar}(c) using Remark~\ref{trans}.

\item[(b)]
The number of perfect matchings in $G$ equals \[\Pm (G)=|\mu(G,0)|.\]
Thus,  \[\frac{\ln \Pm (G)}{v(G)}=\frac{\ln |\mu(G,0)|}{v(G)}=\int\ln|x|\mathrm d\rho_G(x)
.\]
\end{itemize}
\end{proof}

The following theorem deals with the behaviour of the matching measure in a
Benjamini--Schramm convergent graph sequence.

\begin{Th} \label{wc}
Let $(G_n)$ be  a Benjamini--Schramm convergent
 graph sequence with absolute degree bound  $d\ge 2$.  Set  $\omega=2\sqrt{d-1}$ and $K=[-\omega,\omega]$.
Then the sequence of matching measures  $\rho_n=\rho_{G_n}$
converges weakly to a probability measure $\rho$ on $K$. Moreover, we have
\begin{itemize}
\item[(a)]
\[\frac{\ln \mathbb M(G_n)}{v(G_n)}\to 
\frac12 \int_K \ln (1+x^2)\mathrm d\rho(x),\]
\[ \frac{\mathbb E\gamma_{G_n}}{v(G_n)}
\to
\frac12\int_K\frac{x^2}{1+x^2}\mathrm d\rho(x),\]

 \[\frac{\mathbb D^2\gamma_{G_n}}{v(G_n)} 
\to
\frac12\int \frac{x^2}{(1+x^2)^2}\mathrm d\rho(x),\]
  \[\frac{\nu(G_n)}{v(G_n)}\to \frac{1-\rho(\{0\})}2.\] In particular,   $(\ln\mathbb M)/v$, $\mathbb E\gamma/v$, $\mathbb D^2\gamma/v$  and  $\nu/v$ are estimable graph parameters.

\item[(b)]
\[\limsup_{n\to \infty}\frac{\ln \Pm(G_n)}{v(G_n)}\leq \int_K \ln |x|\mathrm d\rho(x).\]

\end{itemize}
\end{Th}

Note that part (a) recovers Theorem~\ref{NOEL}.

\begin{proof} The matching polynomial $\mu(G,x)$ is monic of degree $v(G)$ and multiplicative with respect to disjoint union of graphs. The coefficient of $x^{v(G)-k}$ is of the form $c_kH_k(G)$, where $c_k$ is a constant, $H_k$ is a  graph, and $H_k(G)$ is the number of subgraphs isomorphic to $H_k$ in $G$. These are the only properties we need besides the Heilmann--Lieb Theorem.

 Note that  $\mathbb C\setminus K$ is connected and $K$ has empty  interior (as a  subset of $\mathbb C$).  By  \cite[Theorem 4.6(a)]{csi}, it follows that the sequence $\int_K g\mathrm d\rho_n$  converges for all continuous $g:K\to \mathbb R$, i.e., $\rho_n$ converges  weakly to a measure $\rho$.

(a)
Since $\ln (1+x^2)$ is continuous on $ K$, we have
\[\frac{\ln \mathbb M(G_n)}{v(G_n)}=\frac12 \int_K \ln (1+x^2)\mathrm d\rho_{n}\to\frac12\int_K  \ln (1+x^2) \mathrm d\rho .\]

The statements for the expectation and variance follow from Proposition~\ref{matchpar} in the same way.

From Remark~\ref{trans} and formula~\eqref{delta}, we see that \[\sup_n \rho_n(0<|x|\le \delta)\to 0\] as $\delta\to 0$. Therefore, $\rho_n(\{0\})\to \rho(\{0\})$ as $n\to\infty$.
Thus, the statement for the matching ratio follows from Proposition~\ref{matchpar}.

(b)
 Let   $u(x)=\ln |x|$ and  $u_k(x)=\max (u(x), -k)$ for $k=1,2,\dots$.
Then \[\frac{\ln \Pm(G)}{v(G)}=\int u \mathrm d\rho_G\le \int u_k \mathrm d\rho_G \qquad (k=1,2,\dots). \]
 Thus,
\[\limsup_{n\to\infty}\frac{\ln \Pm(G_n)}{v(G_n)}\le\lim_{n\to\infty}\int_K u_k \mathrm d\rho_{n}= \int_K u_k\mathrm  d\rho \qquad (k=1,2,\dots),\]
 since the measures $\rho_{n}$ are  supported on the  compact interval $K$  not depending on $n$, and  $u_k$ is  continuous and  bounded on $K$.

Since $u_k\ge u_{k+1}$ and $u_k\to  u$ pointwise, the claim follows using the Monotone Convergence Theorem.
\end{proof}

\begin{Rem} An alternative proof for the  weak convergence of $\rho_n$ is possible. Indeed, there is a very nice interpretation of the $k$-th power sum
$p_k(G)$ of the roots of the matching polynomial. It counts the number of
closed tree-like walks of length $k$ in the graph $G$: see chapter 6 of
\cite{god3}.  We don't go into the details of `tree-like walks'; all
we need is that these are special type of walks, consequently we can count them
by knowing all $(k/2)$-balls centered at the vertices of the graph $G$. In
particular, this implies that  for all $k$, the sequence $p_k(G_n)/v(G_n)$ is
convergent, and the weak convergence of $\rho_{n}$ follows.
\end{Rem}

\begin{Rem} One can ask whether estimability of a certain graph parameter actually means that one can get {\it{explicit}} estimates on the parameter from knowing the $R$-neighborhood statistics of a finite graph for some large $R$. In general, this is not clear. However, using Lemma \ref{ize}, one can indeed get such estimates. For instance, when $G$ is $d$-regular and has girth at least $R$, its matching measure has the same first $R$ moments as the matching measure of $H$, where $H$ is a $d$-regular bipartite graph with girth at least $R$. Since  $\rho_H(\{0\})=0$, by Lemma \ref{ize} we get that the matching ratio of $G$ is at least $1/2 - O(1/\ln R)$. Of course, this is a weak estimate but is obtained by purely analytic means.
\end{Rem}

\section{Graphs with large girth} \label{GIRTH}
In this section we study $d$--regular graphs with large girth, in particular, we prove Theorem~\ref{entropy-girth}.
We start by looking at matching measures of regular  graphs with essentially large girth.
Recall that $(G_n)$  essentially has  large girth if, for all $k$, the number of $k$--cycles is $o(v(G_n))$.

\begin{Th} \label{td}
Let
$(G_n)$ be a sequence of $d$--regular graphs with essentially large girth. 
Then $\rho_{G_n}$ converges weakly to the measure with density
function
$$f_d(x)=\frac{d \sqrt{4(d-1)-x^2}}{2\pi
    (d^2-x^2)}\chi_{[-2\sqrt{d-1},2\sqrt{d-1}]}.$$
\end{Th}

From now on, we follow the notations of B.\ McKay. Let
$$\omega=2\sqrt{d-1},$$
and
$$f_d(x)=\frac{d \sqrt{\omega^2-x^2}}{2\pi (d^2-x^2)}\chi_{[-\omega,\omega]}$$
as in Theorem~\ref{td}.
\bigskip

The proof of Theorem~\ref{td} will be an easy application of the following
lemma.

\begin{Lemma}\cite{god3} \label{adj-match} Let $G$ be a graph and let
$\phi(G,x)$ denote the characteristic polynomial of its adjacency
matrix. Let $\mathcal{C}$ denote the set of two-regular subgraphs of $G$,
i.e., these subgraphs are disjoint union of cycles. For $C\in \mathcal{C}$,
let $k(C)$ denote the number of components of $C$. Then
$$\phi(G,x)=\sum_{C\in \mathcal{C}}(-2)^{k(C)}\mu(G\setminus C,x).$$
\end{Lemma}

\begin{proof}[Proof of Theorem~\ref{td}.] First let $(G_n)$ be a graph sequence for
which the girth $g(G_n)\to \infty$.  If $g(G)>k$, then Lemma~\ref{adj-match}
implies that the first $k$ coefficients of $\phi(G,x)$ and $\mu(G,x)$
coincide. This implies that the first $k$ moments of the uniform distribution
arising from the roots of $\phi(G,x)$ and $\mu(G,x)$ coincide too. Since
$g(G_n)\to \infty$, this means that for any fixed $k$, the moments arising from
$\phi(G,x)$ and $\mu(G,x)$ converge to the same limit, actually the two
sequences are the same for large enough values. Then the statement of the
theorem follows from B.\ McKay's work \cite{mck} on the spectral distribution
of random graphs.

In the general case, consider an auxiliary graph sequence $(H_n)$ of $d$--regular graphs such that $g(H_n)\to\infty$. The sequence $G_1$, $H_1$, $G_2$, $H_2$, $\dots$   is Benjamini--Schramm convergent, and the theorem follows using Theorem~\ref{wc}.
\end{proof}

We shall need the following lemma, implicit in McKay's work, on the density function $f_d(x)$.

\begin{Lemma} \label{int-1}
Let $\gamma$ be a complex number that does not lie on either of the real  half-lines $(-\infty,-1/\omega]$ and $[1/\omega,\infty)$. Set $$\eta=\frac{1-\sqrt{1-4(d-1)\gamma^2}}{2(d-1)\gamma^2}=\frac 2{1+\sqrt{1-\omega^2\gamma^2}}.$$ 
Then
\begin{equation}\label{log}\int_{-\omega}^{\omega}f_d(x)\ln (1-\gamma x)\mathrm{d}x=\frac{d-2}{2}\ln \left(
\frac{d-1}{d-\eta}\right)-\ln \eta\end{equation}

and

\begin{equation}\label{reciprok}\int_{-\omega}^{\omega}f_d(x)\frac{\mathrm{d}x}{1-\gamma x}=\frac{d-2-d\sqrt{1-\omega^2\gamma^2}}{2(d^2\gamma^2-1)}=\frac{d-1}{(d/\eta)-1}.\end{equation} We also have \begin{equation}\label{reciproknegyzet}\int_{-\omega}^{\omega}f_d(x)\frac x{(1-\sqrt{-1}x)^2}{\mathrm{d}x}=\frac{4d(d-1)^2\sqrt{-1}}{d(d-2)(4d-3)+(2d^3-d^2-2d+2)\sqrt{4d-3}}.\end{equation}
Note that we use the principal branch of the square root and the logarithm function.
\end{Lemma}

\begin{proof}
Since both sides of \eqref{log} are holomorphic in $\gamma$, we may assume that $|\gamma|<1/\omega$. Then formula~\eqref{log} is a special case of \cite[Lemma~3.5]{mck1983}.

For \eqref{reciprok} we may write $1/(1-\gamma x)=\sum_{i=0}^\infty \gamma^ix^i$.  This reduces the statement to \cite[Theorem~2.3(a)]{mck1983}. Alternatively, we could differentiate \eqref{log} with respect to $\gamma$ and use the fact that $f_d$ is a density function to get \eqref{reciprok}.

The formula \eqref{reciproknegyzet} is the derivative with respect to $\gamma$ of \eqref{reciprok} at the point $\gamma =\sqrt{-1}$. Indeed,
at that point
\[\left(\frac1\eta\right)'=-2\sqrt{-1}\frac{d-1}{\sqrt{4d-3}},\] so \[\left(\frac d\eta-1\right)'=-2\sqrt{-1}\frac{d(d-1)}{\sqrt{4d-3}},\]
while
\[\frac d\eta-1=\frac{d-2+d\sqrt{4d-3}}2,\] so
\[\left(\frac d\eta-1\right)^2=\frac{2d^3-d^2-2d+2+d(d-2)\sqrt{4d-3}}2,\]
whence
\[\sqrt{4d-3}\left(\frac d\eta-1\right)^2=\frac{d(d-2)(4d-3)+(2d^3-d^2-2d+2)\sqrt{4d-3}}2,\]
and the claim follows.
\end{proof}

 \begin{Lemma} \label{int-3}
$$\int_{-\omega}^{\omega}f_d(x)\ln |x|\mathrm dx=\frac12\ln\left(
\frac{(d-1)^{d-1}}{d^{d-2}}\right).$$
\end{Lemma}

\begin{proof}Let $\gamma$ be purely imaginary. We note that \[\ln |1-\gamma x|-\ln |\gamma|\to \ln |x|\] monotonously as $\gamma\to +\infty\sqrt{-1}$. So, in the integral,  we may replace
$\ln |x|$ by this difference and then take the limit.
It is easy to check that $\eta\to 0$ and $|\gamma||\eta|\to 1/\sqrt{d-1}$.  This implies the statement of the lemma using the real part of formula \eqref{log} from the previous lemma.
\end{proof}

We are ready to prove our main result.

\begin{Th} \label{entropy-girth} Let $(G_n)$ be a sequence of $d$--regular
  graphs with essentially large girth. 
\begin{itemize}

\item[(a)] We have
$$
\frac{\ln \mathbb{M}(G_n)}{v(G_n)}\to\frac{1}{2}\ln S_d,$$
where
$$S_d=\frac{1}{\xi^2}\left(\frac{d-1}{d-\xi}\right)^{d-2},\qquad\xi=\frac{\sqrt{4d-3}-1}{2(d-1)}=\frac2{1+\sqrt{4d-3}}.$$ In particular, $S_3=16/5$.

For the expected size of a uniform random matching we have
\[\frac{\mathbb E\gamma_{G_n}}{v(G_n)}\to\frac{{d}}{2}\cdot\frac{1-\xi}{d-\xi}.\]
For $d=3$, this limit is $3/10$.

For the variance, we have \[\frac{\mathbb D^2\gamma_{G_n}}{v(G_n)}\to\frac{d(d-1)^2}{d(d-2)(4d-3)+(2d^3-d^2-2d+2)\sqrt{4d-3}}.\] For $d=3$, this limit is $2/25$.

\item[(b)] For the number of perfect matchings $\Pm(G_n)$, we have
$$\limsup_{n\to \infty}\frac{\ln \Pm(G_n)}{v(G_n)}\leq \frac{1}{2}\ln \left(
\frac{(d-1)^{d-1}}{d^{d-2}}\right).$$

\item[(c)] If, in addition, the graphs $(G_n)$ are bipartite, then
$$\lim_{n\to \infty}\frac{\ln \Pm(G_n)}{v(G_n)}=\frac{1}{2}\ln \left(
\frac{(d-1)^{d-1}}{d^{d-2}}\right).$$
\end{itemize}
\end{Th}

\begin{proof}
 (a)
By Theorem~\ref{td} we know that $\rho_{G_n}$ converges weakly to the measure
$\rho$ with density function $f_d(x)$.  Put $\gamma=\sqrt{-1}$ into Lemma~\ref{int-1}, then $\eta$ becomes the $\xi$ of  Theorem~\ref{entropy-girth} and we are done by  Theorem~\ref{wc}(a), since
\[\frac12\ln (1+x^2)=\Re\ln (1-\sqrt{-1}x),\]
\[\frac{x^2}{1+x^2}=1-\Re \frac1{1-\sqrt{-1}x},\] and \[\frac{x^2}{(1+x^2)^2}=\frac{1}2\Im\frac x{(1-\sqrt{-1}x)^2}.\]

\medskip

\noindent (b) This statement immediately follows from Theorem~\ref{wc}(b)
and Lemma~\ref{int-3}.
\medskip

\noindent (c) The claim follows from part (b) and Schrijver's theorem (Theorem~\ref{Schrijver}).
\end{proof}

\begin{Rem} Friedland's Lower Matching Conjecture (LMC) \cite{FKM} asserts that if $G$ is a $d$--regular bipartite graph on $v(G)=2\cdot n$ vertices and $m_k(G)$ denotes the number of $k$-matchings as before, then
$$m_k(G)\geq {n \choose k}^2\left(\frac{d-t}{d}\right)^{(d-t)n}(td)^{tn},$$
where $t=k/n$.  This is still open, but an  asymptotic version,   Friedland's Asymptotic Lower Matching Conjecture, was proved by L.\ Gurvits \cite{gur2}. Namely, if $G$ is a $d$--regular bipartite graph on $2\cdot n$ vertices, then
$$\frac{\ln m_k(G)}{n}\geq t\ln \left(\frac{d}{t}\right)+(d-t)\ln \left(1-\frac{t}{d}\right)-2(1-t)\ln (1-t)+o_n(1),$$
where $t=k/n$. Note that Gurvits's result implies that for any $d$--regular bipartite graph $G$, we have
$$\frac{\ln \mathbb{M}(G)}{v(G)}\geq \frac{1}{2}\ln S_d.$$
Indeed, the maximum of the function
$$t\ln \left(\frac{d}{t}\right)+(d-t)\ln \left(1-\frac{t}{d}\right)-2(1-t)\ln (1-t)$$
is $\ln S_d$, and if we apply the statement to the  disjoint union 
 of many  copies of the graph $G$, then the $o_n(1)$ term will disappear.
This shows that large girth graphs have an asymptotically  minimal number of matchings among bipartite graphs.
\end{Rem}

\section{Negative results: perfect matchings of bipartite graphs} \label{perfect_entropy}
It is easy to show that the \it perfect matching entropy per vertex, \rm defined as
$$\frac{\ln \Pm(G)}{v(G)},$$
is not an estimable graph parameter since sampling cannot distinguish  a
large graph with many perfect matchings from the  same  large
graph with an  additional isolated vertex (which has no perfect
matchings). We shall  show that even if we consider $d$--regular
bipartite graphs, the situation does not change.

\medskip

\noindent\bf Notation. \rm Given a  finite graph $G$ admitting at least one perfect matching, and given an edge $e$ of $G$, let
$p(e)$ be the probability that a uniform random perfect matching contains $e$.

\begin{Cons}\label{constr} Let $G$ be a $d$--regular bipartite graph. Recall the following  well-known construction of an  $n$-fold cover of
$G$.

Consider $n$ disjoint copies of $G$,
erase all $n$  copies of the edge $e=\{x,y\}$, and restore $d$--regularity by adding  $n$ new edges: connect each copy of $x$ to the copy of $y$ in the (cyclically) next copy of $G$. This gives us  a graph $G'$.

We now calculate how edge probabilities are transformed. If we had
$p(e)=1/(x+1)$ in $G$, then, for each new edge $e'$, we
have \[p(e')=\frac{p(e)^n}{p(e)^n+(1-p(e))^n}=\frac 1{x^n+1}\] in $G'$. This is because any perfect matching of $G'$ contains either all new edges
or none, and the perfect matchings of these two types are in obvious
bijections with $n$-tuples of perfect matchings of $G$ containing,
respectively not containing $e$.

Let
$f\in E(G)$ be an edge adjacent to $e$. Let $f'\in E(G')$ be the corresponding
edge adjacent to $e'$.  Then
\[p(f')=p(f)\frac{1-p(e')}{1-p(e)}=p(f)\frac{x^{n-1}(x+1)}{x^n+1}.\]

For $i=1,2,\dots, d$, define a  map $T_i^{(n)}$ from $(0,1]^d$ to itself as
  follows. If a vector has  $i$-th coordinate equal to $1/(x+1)$ and $j$-th
  coordinate equal to $y_j$ for $j\ne i$, then the image will have   $i$-th
  coordinate equal to $1/(x^n+1)$ and $j$-th coordinate equal to
  $y_j{x^{n-1}(x+1)}/({x^n+1})$ for $j\ne i$.

Let $f_1$, \dots, $f_d$ be the edges emanating from one end of $e$, so that
$f_i=e$ for one index $i$.  Consider  the corresponding edges $f_1'$, \dots,
$f_d'$ emanating from one of the two  corresponding vertices of $G'$, one of
these edges being the new edge $f_i'=e'$. Then
\[(p(f_1'),\dots, p(f_d'))=T_i^{(n)}(p(f_1),\dots, p(f_d)).\]
\end{Cons}

We wish to construct regular bipartite graphs with irregular edge
probabilities. As a warm-up, we construct a  graph that has a very improbable
edge. This will not be formally needed for the sequel.

\begin{Th}
\begin{itemize}
\item[(a)] For any integers $d\ge 1$ and $n\ge 0$, there exists a
  $d$--regular bipartite graph on $2\cdot n$ points with an edge $e$ such
  that
\[p(e)=\frac 1{(d-1)^{n}+1}.\]

\item[(b)] For any integers $d\ge 1$ and $n\ge 0$, there exists a $d$--regular
  bipartite simple graph on $2\cdot d n$ points with an edge $e$ such
  that $p(e)$ is as in (a) above.
\end{itemize}
\end{Th}

\begin{proof}
 Apply Construction~\ref{constr}  starting from the $d$--regular bipartite
graph on two points for part (a) and from  $K_{d,d}$ for part (b).
\end{proof}

We now take up the opposite task: producing a high edge probability.

\begin{Lemma} For any integer $d\ge 3$,  there exists  a $d$--regular bipartite
  simple graph  with an edge $e$ such that $p(e)>1/2.$
\end{Lemma}

\begin{proof}
It suffices to prove that there exist positive integers $n_2$, $n_3$, \dots,
$n_d$ such that the first coordinate of
\[T_2^{(n_2)} T_3^{(n_3)}\cdots
T_d^{(n_d)}\left(\frac 1d,\dots,\frac 1d\right)\]
is larger than $1/2$.

 For this, we use induction on $d$.  For $d=3$, a direct calculation shows
 that the first coordinate of
 \[T_2^{(3)} T_3^{(2)}\left(\frac 13,\frac13,\frac 13\right)\] is ${18}/{35}>1/2$.

Suppose that for some $d\ge 3$, the  positive integers $n_2$, $n_3$, \dots,
$n_d$ have the desired property. Note that
\[T_{d+1}^{(n)}\left(\frac 1{d+1},\dots,\frac 1{d+1},\frac
1{d+1}\right)\to\left(\frac 1d,\dots,\frac 1d,0\right)\]
as $n\to\infty$. Note also that all maps $T_i^{(n)}$ (in dimension $d$) are
continuous. 
Thus, if $n_{d+1}$
is large enough, then the sequence $n_2$, $n_3$, \dots, $n_d$, $n_{d+1}$ will
have the desired property for $d+1$.
\end{proof}

We now prove Theorem~\ref{construction}; we repeat the statement.
\bigskip

\noindent \textbf{Theorem~\ref{construction}.} \textit{ For any integer $d\ge 3$, there exists a constant $0<c<1$ such that
  for any positive integer $n\ge d$ there exists  a $d$--regular bipartite
  simple graph on $2\cdot n$ points with an edge $e$ such that \[p(e)>1-c^n.\]}

\begin{proof}
Let $G_0$ and $e_0$ be the graph and the edge given by the Lemma, with $G_0$
having $2\cdot n_0$ vertices. We can write \[p(e_0)=\frac 1{1+c^{n_0}},\] where
$0<c<1$. If $n=rn_0$ with an  integer $r$,  then we can apply
 Construction~\ref{constr} with $r$ in place of $n$ 
to get a $d$--regular
bipartite simple graph on $2\cdot n$ vertices, having an edge $e$ such
that
 \[p(e)=\frac 1{1+c^{n_0r}}=\frac 1{1+c^{n}}>1-c^n,\] so the statement
 holds for all  $n$ divisible by  $n_0$.  By changing $c$ if necessary,
 we can achieve that it hold for all integers $n\ge d$.
\end{proof}

Starting from a graph $G$ as in this  Theorem, we shall produce two
$d$--regular bipartite graphs (denoted $2G$ and $\tilde G$) that share a
common  vertex set and  differ only in two edges, but have a  very different
number of perfect matchings.
\bigskip

\noindent \textbf{Theorem~\ref{perfect_matching}.} \textit{Fix $d\ge 3$. Then, for $d$--regular bipartite simple graphs,  the
graph parameter $(\ln \Pm)/v$ is not estimable.}
\bigskip

\begin{proof}
For $n\ge d$, let $G=G_n$ and $e=e_n$ be the graph and the edge given by the
previous Theorem, with $G$ having $2\cdot n$ vertices and $p(e)>1-c^n$. Let $f$ be
an edge adjacent to $e$, so that $p(f)\le 1-p(e)<c^n$. Let $2G=2G_n$ denote
the disjoint union of two copies of $G$, so that $\Pm(2G)=\Pm(G)^2$. 
Let $\tilde G=\tilde G_n$ be the
graph $2G$ with $e$ erased from the first copy of $G$ and $f$ erased from the
second copy, and with $d$--regularity restored by two new edges going
across. We have
\begin{align*}
\Pm(\tilde G)=\Pm(G)^2\left(p(e)p(f)+(1-p(e))(1-p(f))\right)\le\\\le
\Pm(2G)(p(f)+1-p(e))<\Pm(2G)\cdot 2c^{n}.\end{align*}
Thus, \[\frac {\ln \Pm(2G)}{v(2G)}-\frac {\ln \Pm(\tilde G)}{v(\tilde
  G)}>-\frac{\ln ( 2c^{n})}{4n}\to \frac{1}{4}\ln\frac{1}{c}>0\]
as $n\to\infty$.

Choose a Benjamini--Schramm convergent subsequence $(G_{n_k})$ of the sequence
$G_d$, $G_{d+1}$, \dots.
Then the graph sequence $2G_{n_1}$, $\tilde G_{n_1}$, $2G_{n_2}$, $\tilde
G_{n_2}$, \dots is also Benjamini--Schramm convergent. The graph parameter
$(\ln \Pm)/v$ does not converge along this graph sequence and therefore is not estimable.
\end{proof}

\section{Perfect matchings in bipartite expander graphs} \label{expander_helps}

In this part we prove Theorem~\ref{expander}. We say that the bipartite graph $G=(U,V,E)$ with vertex classes $U$ and $V$ is a
$\delta$-expander if
$$|N(U')| \geq (1+\delta)|U'|$$
holds for every
$U' \subset U$ such that $ |U'| \leq |U|/2$, and
$$|N(V')| \geq (1+\delta)|V'|$$
holds for every $V' \subset V$ such that $ |V'| \leq |V|/2$.

A digraph $G=(V,E)$ is a $\delta$-expander if for every $V' \subset V(G), |V'|
\leq |V|/2$ the inequalities
$$|N_{in}(V')| \geq (1+\delta)|V'|\ \ \mbox{and}\ \  |N_{out}(V')|  \geq
(1+\delta)|V'|$$
hold. Given a graph $G$, a matching and an even cycle in $G$, we call a cycle
\textit{alternating} if every other edge of the cycle is in the matching.

\begin{Lemma}
Let $G=(U, V, E)$ be a bipartite graph and $M$ a perfect matching of $G$.
Consider the following digraph $\ovl{G}$: $V(\ovl{G})=V$ and
$(x,y) \in E(\ovl{G})$ if and only if there exists $u \in U$ such that $(x,u)
\in M$  and $(u,y) \in E(G)$. Then, for
every $\delta>0$, if $G$ is a $\delta$-expander, then $\ovl{G}$ is a
$\delta$-expander.
\end{Lemma}

\begin{proof}
Let $V' \subseteq V$; assume that $V' \leq |V|/2$. Since $G$ is a $\delta$-expander, we have
$|N(V')| \geq (1+\delta)|V'|$. And the set
of vertices matched to $N(V')$ has size at least $|N(V')|$, hence
$|N_{in}(V')| \geq (1+\delta)|V'|$. The size of $N_{out}(V')$ can be estimated
similarly.
\end{proof}

\begin{Lemma} Let $\delta>0$.
Let $\ovl{G}$ be a $\delta$-expander digraph on $n$ vertices. Then every edge of $\ovl{G}$ is
contained in a directed cycle of length at most
$$1 + 2 \frac{\ln n}{\ln(1+\delta)}.$$
\end{Lemma}

\begin{proof} We may assume $n\ge 2$.
Set
$$k=1+\left\lfloor\frac{\ln\lfloor n/2\rfloor}{\ln(1+\delta)}\right\rfloor.$$
Let $(x,y)$ be an arbitrary edge. Consider the following sets of vertices
defined recursively:
$$S_0=\{ x \}, T_0=\{ y \}\ \ \mbox{and}\ \  S_{i+1}=N_{in}(S_i),
T_{i+1}=N_{out}(T_i),$$
for $i=0, \dots , (k-1)$.
The expander property implies that
$$|S_i|, |T_i| \geq \min \{ (1+\delta)^i ; (1+\delta)\lfloor n/2\rfloor \}=(1+\delta)^i$$   for $i=0, \dots , k$.
This yields $|S_k|, |T_k| > n/2$, hence the sets $S_k$ and $T_k$
are not disjoint. There is a closed directed walk of length $(2k+1)$ through $(x,y)$,
and so there is a directed cycle of length at most $(2k+1)$ through $(x,y)$.  Using that $n\ge (1+\delta)\lfloor n/2\rfloor$, the lemma follows.
\end{proof}

\begin{Lemma}
Let $n\ge 2$,  $\delta>0$, let $G$ be a $\delta$-expander bipartite graph on $2\cdot n$ vertices, and $M$ a
perfect matching of $G$. Then every edge of $G$ is contained by an alternating
cycle of length at most
$$2 + 4 \frac{\ln n}{\ln(1+\delta)}.$$
\end{Lemma}

\begin{proof}Let $G=(U, V, E)$.
Consider the following digraph $\ovl{G}$: $V(\ovl{G})=V$ and
$(x,y) \in E(\ovl{G})$
if and only if  there exists $u \in U$ such that $(x,u) \in M$ and $(u,y) \in E(G)$.
Let $e$ be an edge of $G$, and consider an edge of $\ovl{G}$ that corresponds
to a $2$-path containing $e$. We know that $\ovl{G}$ is a $\delta$-expander,
so it has a directed cycle of length at most
$$1 +2\frac{\ln n}{\ln(1+\delta)}$$
containing this edge. This cycle
will correspond to an alternating cycle of the required length.
\end{proof}

Now we prove Theorem~\ref{expander}. We repeat the statement.
\bigskip

\noindent \textbf{Theorem~\ref{expander}} \textit{
Let $n\ge 2$, $\delta>0$, let $G$ be a $\delta$-expander bipartite graph of maximum
degree $d$ on $2\cdot n$ vertices, and $e$ an edge of $G$. Assume that $G$ admits a perfect matching. 
Then
$$
p(e) \geq \frac 1dn^{-2\cdot { \ln(d-1)}/{\ln(1+\delta)}}.$$}

\begin{proof}
Consider the following bipartite graph $H=(A, B, E(H) )$: The elements of the
set $A$  are the perfect matchings of $G$ not containing $e$, while the
elements of the set $B$  are the perfect matchings of $G$ containing $e$. The
pair $(M,N)$ is in $E(H)$  if the symmetric difference of $M$ and $M'$ is a
cycle of length at most
$$2 + 4 \frac{\ln n}{\ln(1+\delta)}.$$
Note that $p(e ) = \frac{|B|}{|A|+|B|}$.
Every matching $M$ in $A$ has degree at least one,
since there is an alternating cycle of length at most
$$2 + 4 \frac{\ln n}{\ln(1+\delta)}$$
containing $e$. On the other hand, given a matching in $B$, the edge $e$ is
contained by at most
$$(d-1)^{1 + 2 \cdot \ln n/\ln(1+\delta)}$$
alternating cycles of length at most
 $$2 + 4 \frac{\ln n}{\ln(1+\delta)}.$$
Hence the maximum  degree of $B$ can be at most $(d-1)^{1 + 2 \cdot \ln n/\ln(1+\delta)}$.
The theorem follows.
\end{proof}

We end this section with the following simple observation on the corresponding problem concerning all matchings.

\begin{Prop} Let $G$ be a graph with maximum degree $d$, $e \in E(G)$ and $M$ a matching in $G$ chosen uniformly at random. Then
$$\mathbb{P}(e \in M) \geq \frac{1}{d^2+1}.$$
\end{Prop}

\begin{proof}
To every matching, we assign another matching that contains $e$:
Given an arbitrary matching we remove the edges of the matching adjacent
to $e$ and add $e$ to the matching. The pre-image of every matching containing $e$ consists of at most $(d^2+1)$ matchings.
\end{proof}

\section{Open problems} \label{we don't know}

There are two natural questions arising from the previous sections.

\begin{?} Is it true that if $(G_n)$ is a sequence of $d$--regular bipartite
  graphs such that
$$\frac{\ln \Pm(G_n)}{v(G_n)}\to \frac{1}{2}\ln
  \left(\frac{(d-1)^{d-1}}{d^{d-2}}\right),$$
then 
 $G_n$  essentially has large girth?
\end{?}

Given a graph $G$, set $p_{\min}(G)=\min_{e \in E(G)} p(e)$ and
$p_{\max}(G)=\max_{e \in E(G)} p(e)$.

\begin{?}
Can $p_{\min}(G)$ be arbitrarily close to $0$, resp.\ to 1,  for $\delta$-expander graphs with  degree bound $d$ ?
What is the expected value of $p_{\min}$ and $p_{\max}$ for random regular graphs?
\end{?}

 \bigskip \noindent \textbf{Acknowlegment.} The authors thank  the anonymous referee for helpful comments. The second author is very grateful to D.\ Gamarnik for a useful discussion.


\begin{thebibliography}{99}

\bibitem{abert_hubai} M.\ Ab\'ert, T.\ Hubai: \textit{Benjamini--Schramm
  convergence and the distribution of chromatic roots for sparse graphs},
arXiv:1201.3861v1, to appear in Combinatorica

\bibitem{abthvi} M.\ Ab\'ert, A.\ Thom and B. Vir\'ag: \textit{Benjamini-Schramm convergence and pointwise convergence of the spectral measure}, preprint at http://www.math.uni-leipzig.de/MI/thom/

\bibitem {bensch}I. Benjamini and O. Schramm, \textit{Recurrence of distributional
limits of finite planar graphs}, Electron. J. Probab. 6 (2001), no. 23, 13 pp.

\bibitem{bollobas} B. Bollob\'{a}s, \textit{The independence ratio of
regular graphs}, Proc. Amer. Math. Soc 83 (1981) no.2 433--436.

\bibitem{bolmc} B.\ Bollob\'as and B.\ D.\ McKay: \textit{The number of matchings in
  random regular graphs and bipartite graphs}, J.\ Combinatorial Theory, Series
  B \textbf{41} (1986), pp. 80-91

\bibitem{chu} M.\ Chudnovsky and P.\ Seymour: \textit{The roots of the
  independence polynomial of a clawfree graph},   J.\ Combin.\ Theory
  Ser.\  B \textbf{97}(3) (2007), pp. 350--357.

\bibitem{BGK+}  M.\ Bayati, D.\ Gamarnik, D.\ Katz, C.\ Nair and P.\ Tetali, \textit{Simple deterministic algorithm for counting matchings}, ACM Symp. on Theory of Computing (STOC, San Diego, 2007, appeared)

\bibitem{csi}  P.\ Csikv\'ari and P.E.\ Frenkel, \textit{Benjamini--Schramm
  continuity of root moments of graph polynomials}, ArXiv preprint
  1204.0463

\bibitem{ele}  G.\ Elek and G.\ Lippner: \textit{Borel oracles. An analytical
  approach to constant-time algorithms},
  Proc.\ Amer.\ Math.\ Soc.\ \textbf{138}(8) (2010), pp. 2939--2947.

\bibitem{erd}
P.\ Erd\H os and H.\ Sachs: \textit{Regul\"are Graphen gegebener Taillenweite mit
minimaler Knotenzahl},
Wiss.\ Z.\ Univ.\ Halle, Math.-Nat. \textbf{12}(3) (1963), pp.\ 25--258.

\bibitem{FKM} S.\ Friedland, E.\ Krop and K.\ Markstr\"om:  \textit{On the Number of Matchings in Regular Graphs}, The Electronic Journal of Combinatorics, \textbf{15} (2008), R110, pp. 1--28.

\bibitem{GK}  D.\ Gamarnik and D.\ Katz: \textit{A deterministic approximation algorithm for computing the permanent of a 0, 1 matrix},
 J. Comput. System Sci. 76 (2010)(8), pp.\ 879--883.

\bibitem{god3} C.\ D.\ Godsil: \textit{Algebraic Combinatorics}, Chapman and
Hall, New York 1993

\bibitem{gur} L.\ Gurvits: \textit{Van der Waerden/Schrijver-Valiant like
  conjectures and stable (aka hyperbolic) homogeneous polynomials: one theorem
  for all},  Electron. J. Combin.  \textbf{15}(1)  (2008), Research Paper 66

\bibitem{gur2} L.\ Gurvits: \textit{Unleashing the power of Schrijver's
permanental inequality with the help of the Bethe Approximation}, ArXiv preprint 1106.2844v11

\bibitem{hei}  O.\ J.\ Heilmann and E. H. Lieb: \textit{Theory of monomer-dimer
systems}, Commun. Math. Physics \textbf{25} (1972), pp. 190--232

\bibitem{L-S} M.\ Laurent and A.\ Schrijver: \textit{On Leonid Gurvits's proof for
  permanents}, Amer. Math. Monthly  \textbf{117}(10)  (2010), 903--911.

\bibitem{luck}W. L\"{u}ck: \textit{$L^{2}$-invariants: theory and
applications to geometry and $K$-theory}, A Series of Modern Surveys in
Mathematics, 44. Springer-Verlag, Berlin, 2002

\bibitem{lyons} R. Lyons: \textit{Asymptotic enumeration of spanning trees},
  Combinatorics, Probability and Computing \textbf{14} (2005), pp. 491-522

\bibitem{mck} B.\ D.\  McKay: \textit{The expected eigenvalue distribution of
a large regular graph}, Linear Algebra and its Applications \textbf{40}
  (1981), pp. 203--216

\bibitem{mck1983} B.\  D.\  McKay: \textit{Spanning trees in regular graphs}, European J. Combinatorics
\textbf{4} (1983), 149--160.

\bibitem{ngu} H.\ N.\ Nguyen and K.\ Onak, \textit{Constant-time approximation
  algorithms via local improve ments}, 49th Annual IEEE Symposium on
  Foundations of Computer Science (2008), pp. 327--336.

\bibitem{sch1} A.\ Schrijver: \textit{Counting 1-factors in regular bipartite
  graphs}, J. Combin. Theory Ser. B \textbf{72} (1998), pp. 122--135.

\bibitem{sch2} A.\ Schrijver and W.\ G.\ Valiant: \textit{On lower bounds for
  permanents}, Mathematics Proceedings A \textbf{83} (4) (1980), pp. 425--427

\bibitem{vor} M.\ Voorhoeve: \textit{A lower bound for the permanents of
  certain (0,1)-matrices}, Indagationes Mathematicae \textbf{41} (1979),
  pp. 83--86

\bibitem{wan} I.\ Wanless:  \textit{Addendum to Schrijver's work on minimum
  permanents}, Combinatorica \textbf{26}(6) (2006), pp. 743--745

\end{thebibliography}
\end{document}